\documentclass[12pt,a4paper]{amsart}
\usepackage{amsfonts,color}
\usepackage{amsthm}
\usepackage{amsmath}
\usepackage{amscd}
\usepackage[latin2]{inputenc}
\usepackage{t1enc}
\usepackage[mathscr]{eucal}
\usepackage{indentfirst}
\usepackage{graphicx}
\usepackage{graphics}
\usepackage{pict2e}
\usepackage{epic}
\usepackage{url}
\usepackage{epstopdf}
\usepackage{comment}

\usepackage{tikz}
\usetikzlibrary{calc,decorations.pathreplacing}
\usetikzlibrary{quotes,angles}

\tikzset{point/.style={circle,fill,draw,inner sep=0,minimum size=1pt}}
\tikzset{vertex/.style={circle,fill,draw,inner sep=0,minimum size=7pt}}
\tikzset{overtex/.style={circle,fill=none,draw,inner sep=0,minimum size=7pt}}

\numberwithin{equation}{section}
\usepackage[margin=2.6cm]{geometry}

\allowdisplaybreaks

\theoremstyle{plain}
\newtheorem{Th}{Theorem}[section]
\newtheorem{Lemma}[Th]{Lemma}
\newtheorem{Cor}[Th]{Corollary}

 \theoremstyle{definition}
\newtheorem{Def}[Th]{Definition}

\newtheorem{Rem}[Th]{Remark}
\newtheorem{?}[Th]{Problem}

\newcommand{\E}{\mathbb{E}}
\newcommand{\ch}{\mathrm{ch}}

\begin{document}

\title{Evaluations of  Tutte polynomials of regular graphs}

\author[F. Bencs]{Ferenc Bencs}

\address{Alfr\'ed R\'enyi Institute of Mathematics, H-1053 Budapest Re\'altanoda utca 13-15}

\email{ferenc.bencs@gmail.com}

\author[P. Csikv\'ari]{P\'{e}ter Csikv\'{a}ri}

\address{Alfr\'ed R\'enyi Institute of Mathematics, H-1053 Budapest, Re\'altanoda utca 13-15}

\email{peter.csikvari@gmail.com}

\thanks{The first author is supported by the NKFIH (National Research, Development and Innovation Office, Hungary) grant KKP-133921.
The second author  is partially supported by the  Counting in Sparse Graphs Lend\"ulet Research Group.
}

 \subjclass[2010]{Primary: 05C30. Secondary: 05C31, 05C70}

 \keywords{forests, Tutte polynomial} 

\begin{abstract} Let $T_G(x,y)$ be the Tutte polynomial of a graph $G$. In this paper we show that if $(G_n)_n$ is a sequence of $d$-regular graphs with girth $g(G_n)\to \infty$, then for $x\geq 1$ and $0\leq y\leq 1$ we have
$$\lim_{n\to \infty}T_{G_n}(x,y)^{1/v(G_n)}=t_d(x,y),$$
where
$$t_d(x,y)=\left\{\begin{array}{lc} (d-1)\left(\frac{(d-1)^2}{(d-1)^2-x}\right)^{d/2-1}&\ \ \mbox{if}\ x\leq d-1,\\
x\left(1+\frac{1}{x-1}\right)^{d/2-1} &\ \ \mbox{if}\ x> d-1.
\end{array}\right.$$
independently of $y$ if $0\leq y\leq 1$.
If $(G_n)_n$ is a sequence of random $d$-regular graphs, then the same statement holds true asymptotically almost surely.

This theorem generalizes results of McKay ($x=1,y=1$, spanning trees of random $d$-regular graphs) and Lyons ($x=1,y=1$, spanning trees of large-girth $d$-regular graphs). Interesting special cases are $T_G(2,1)$ counting the number of spanning forests, $T_G(2,0)$ counting the number of acyclic orientations.
\end{abstract}

\maketitle

\section{Introduction}

Let $G=(V,E)$ be a graph with $v(G)$ vertices and $e(G)$ edges. For a graph $G=(V,E)$ let $T_G(x,y)$ denote its Tutte polynomial \cite{tutte1954contribution}, that is,
$$T_G(x,y)=\sum_{A\subseteq E}(x-1)^{k(A)-k(E)}(y-1)^{k(A)+|A|-v(G)},$$
where $k(A)$ denotes the number of connected components of the graph $(V,A)$, and $v(G)$ denotes the number of vertices of the graph $G$. Tutte actually defined $T_G(x,y)$ as a sum for spanning trees in terms of internally and externally active edges of the spanning trees. This definition shows that $T_G(x,y)$ has non-negative coefficients. 

The Tutte polynomial encodes a lot of quantitative information about the graph $G$. For instance, $\tau(G)=T_G(1,1)$ and $F(G)=T_G(2,1)$ count the number of spanning trees and spanning forests, respectively. By a spanning forest $F$, we mean an acyclic set of edges, and the word spanning simply refers to the fact that we always consider $V(G)$ to be the vertex set of the forest $F$. Similarly, $a(G)=T_G(2,0)$ counts the number of acyclic orientations of $G$, and $T_G(0,2)$ counts the number of strongly connected orientations. Other special cases include the chromatic polynomial $\ch(G,q)=(-1)^{v(G)-k(G)}q^{k(G)}T_G(1-q,0)$ counting the number of proper colorings of the graph $G$ with $q$ colors, and the flow polynomial $C_G(q)=(-1)^{e(G)-v(G)+k(G)}T_G(0,1-q)$ counting the nowhere-zero $\mathbb{Z}_q$-flows. For a comprehensive introduction to the theory of the Tutte polynomial. we recommend  the surveys \cite{brylawski1992tutte,crapo1969tutte,welsh1999tutte,ellis2011graph}. 

In this paper, we study the asymptotic behaviour of the Tutte polynomial of regular graphs of large girth that is analogous to the following theorems due to McKay and Lyons. For a graph $G$ let $g(G)$ denote the girth of the graph, that is, the length of the shortest cycle of $G$. 

\begin{Th}[McKay \cite{mckay1981spanning2}] \label{tree-limit1}
Let $(G_n)_n$ be a sequence of $d$-regular random graphs. Then asymptotically almost surely we have 
$$\lim_{n\to \infty}\tau(G_n)^{1/v(G_n)}=\frac{(d-1)^{d-1}}{(d^2-2d)^{d/2-1}}.$$
\end{Th}

\begin{Th}[Lyons \cite{lyons2002asymptotic}] \label{tree-limit}
Let $(G_n)_n$ be a sequence of connected $d$-regular graphs such that  $\lim_{n\to \infty} g(G_n)=\infty$. Then
$$\lim_{n\to \infty}\tau(G_n)^{1/v(G_n)}=\frac{(d-1)^{d-1}}{(d^2-2d)^{d/2-1}}.$$
\end{Th}

Lyons actually proved a more general result on Benjamini--Schramm convergent graph sequences.

Our main theorem is an analogue of these theorems for evaluations of the Tutte polynomial for a wide range of parameters.

\begin{Th} \label{tutte-limit}
Let $x\geq 1$ and $0\leq y\leq 1$. Let $d\geq 2$, and let $(G_n)_n$ be a sequence of $d$-regular graphs such that  $\lim_{n\to \infty} g(G_n)=\infty$. Then
$$\lim_{n\to \infty}T_{G_n}(x,y)^{1/v(G_n)}=t_d(x,y),$$
where
$$t_d(x,y)=\left\{\begin{array}{lc} (d-1)\left(\frac{(d-1)^2}{(d-1)^2-x}\right)^{d/2-1}&\ \ \mbox{if}\ x\leq d-1,\\
x\left(1+\frac{1}{x-1}\right)^{d/2-1} &\ \ \mbox{if}\ x> d-1.
\end{array}\right.$$
independently of $y$ if $0\leq y\leq 1$.
If $(G_n)_n$ is a sequence of random $d$-regular graphs, then the same statement holds true asymptotically almost surely. In fact, if $L(G,g)$ denotes the number of cycles of length at most $g-1$ in a graph $G$, then the same conclusion holds if for every fixed $g$ we have $\lim_{n\to \infty}\frac{L(G_n,g)}{v(G_n)}=0$.
\end{Th}

Note that the case $x=1, y=1$ covers Theorems~\ref{tree-limit1} and \ref{tree-limit}. In case of $d=2$, one needs to define $t_2(1)=1$. Note that for cycle $C_n$ on $n$ vertices we have $T_{C_n}(x,1)=\frac{x^n-1}{x-1}$ which shows that $\lim_{n\to \infty}T_{C_n}(x,1)^{1/n}=t_2(x)$ even if $0\leq x<1$. We believe that that the statement of Theorem~\ref{tutte-limit} is true for every $d$ when $0\leq x<1$  and $0\leq y\leq 1$ except $x=y=0$, but our proof does not work in this case. Naturally, we can introduce the function $t_d(x,y)$ for every $x$ and $y$ the way it is defined in Theorem~\ref{tutte-limit}, but it is not a priori clear that it exists. The authors conjecture that it exists if $x,y> 0$ and if the graphs $G_n$ do not contain loops and bridges, then we can also allow $x=0$ or $y=0$. Let us mention that $t_d(x,y)$ will depend on $y$ if $y$ is large enough in terms of $x$. 
\bigskip

For the number of spanning forests ($F(G)=T_G(2,1)$) and acyclic orientations ($a(G)=T_G(2,0)$)  we immediately get the following statement.

\begin{Th} \label{forest-limit}
Let $(G_n)_n$ be a sequence of $d$-regular graphs such that  $\lim_{n\to \infty} g(G_n)=\infty$. Let $F(G)$ denote the number of spanning forests of the graph $G$. Similarly, let $a(G)$ denote the number of acyclic orientations of the graph $G$. Then
$$\lim_{n\to \infty}F(G_n)^{1/v(G_n)}=\lim_{n\to \infty}a(G_n)^{1/v(G_n)}=\frac{(d-1)^{d-1}}{(d^2-2d-1)^{d/2-1}}.$$
If $(G_n)_n$ is a sequence of random $d$-regular graphs, then the same statement holds true asymptotically almost surely.  In fact, if $L(G,g)$ denotes the number of cycles of length at most $g-1$ in a graph $G$, then the same conclusion holds if for every fixed $g$ we have $\lim_{n\to \infty}\frac{L(G_n,g)}{v(G_n)}=0$.
\end{Th}

We note that the special case of Theorem~\ref{forest-limit} ,when $d=3$, was previously proved by Borb\'enyi, Csikv\'ari and Luo \cite{borbenyi2020number}.
\bigskip

\textbf{Related works.} In statistical physics, there is a huge literature on the convergence of $Z_{G_n}^{1/v(G_n)}$ for various graph sequences $(G_n)_n$ and graph parameters $Z_G$. Most often, $Z_G$ is the partition function of some statistical physical model, and $(G_n)_n$ converges to the lattice $\mathbb{Z}^d$ or another Archimedean lattice. In these cases, it is generally easy to establish the convergence of $Z_{G_n}^{1/v(G_n)}$, but often hard to compute the explicit limit. In the case when $(G_n)_n$ is a sequence of large girth graphs, there is no real-life physical interpretation, but statistical physicists actively study this case, too. In this case, the limit object is the infinite $d$-regular tree, also known as the Bethe lattice. The general convergence concept that is amenable to handle such cases is called Benjamini--Schramm convergence, left convergence, or local weak convergence. In the case of the infinite $d$-regular tree, it is not obvious and in many cases not even true, that $Z_{G_n}^{1/v(G_n)}$ converges. 

Statistical physicists study the Tutte polynomial in the form
$$Z_G(q,w)=\sum_{A\subseteq E(G)}q^{k(A)}w^{|A|},$$
where $q\geq 0$ and $w\geq -1$.
This is the partition function of the random cluster model.
The connection between $T_G(x,y)$ and $Z_G(q,w)$ is
$$T_G(x,y)=(x-1)^{-k(E)}(y-1)^{-v(G)}Z_G((x-1)(y-1),y-1).$$
When $q$ is a positive integer, then $Z_G(q,w)$ is also the partition function of the Potts model that can be described as follows: let $B$ be the $q\times q$ matrix with diagonal elements $1+w$ and off-diagonal elements $1$. For a map $\sigma: V(G)\to [q]$, where $[q]=\{1,2,\dots ,q\}$, let
$$w(\sigma)=\prod_{(u,v)\in E(G)}B_{\sigma(u),\sigma(v)},$$
and
$$Z_G(q,w)=\sum_{\sigma:V(G)\to [q]}w(\sigma).$$

It is known that the random cluster model behaves differently in the regimes $0<q<1$ and $q\geq 1$, and also  the regime $q>1$ and $-1<w<0$  is different from $q>1$ and $w\geq 0$. When $q\geq 1$ and $w\geq 0$, then positive correlation inequality is available through the use of the so-called FKG-inequality \cite{fortuin1971correlation}. When $0<q<1$ and $w\geq 0$, then negative correlation is conjectured \cite{grimmett2004negative, pemantle2000towards}. 
The case $q=2$ and $w\geq 0$ is the so-called ferromagnetic Ising-model. In this case, Dembo and Montanari \cite{dembo2010ising} proved that $Z_{G_n}(2,w)^{1/v(G_n)}$ converges if $(G_n)_n$ is a sequence of $d$-regular large girth graphs. In fact, they proved a significantly more general theorem about not necessarily regular locally tree-like graphs. When $q$ is a positive integer and $w\geq 0$, we get the ferromagnetic Potts model. In this case, Dembo, Montanari and Sun \cite{dembo2013factor} proved the convergence $Z_{G_n}(q,w)^{1/v(G_n)}$ for  $d$-regular large girth graphs for every $d$ first, $q$ is a positive integer and $w\geq 0$, except when $w$  belongs to a certain interval $(w_0,w_1)$. Then, Dembo, Montanari, Sly and Sun \cite{dembo2014replica} proved the convergence of $Z_{G_n}(q,w)^{1/v(G_n)}$ for  $d$-regular large girth graphs, when $d$ is even, $q$ is a positive integer and $w\geq 0$ even if $w\in (w_0,w_1)$. Very recently, Helmuth, Jenssen and Perkins \cite{helmuth2020finite} proved the convergence of  $Z_{G_n}(q,w)^{1/v(G_n)}$ for  $d$-regular large girth graphs for large (not necessarily integer) $q$ and $w\geq 0$ with the additional hypothesis that $(G_n)_n$ satisfies some expansion condition. Ferromagnetic Potts models on random regular graphs are also studied by Galanis, \v{S}tefankovi\v{c}, Vigoda and Yang \cite{galanis2016ferromagnetic}.

When $q$ is a positive integer and $w=-1$, then $Z_G(q,-1)$ counts the number of proper colorings of the graph $G$. This case was treated by Bandyopadhyay and Gamarnik \cite{bandyopadhyay2008counting}.  They showed that if $q\geq d+1$, then for a $d$-regular large girth graph sequence $(G_n)_n$ we have
$$\lim_{n\to \infty}Z_{G_n}(q,-1)^{1/v(G_n)}=q\left(1-\frac{1}{q}\right)^{d/2}.$$
Their result was extended for integer $q\geq 2\Delta$ and $w\geq -1$ by Borgs, Chayes, Kahn and Lov\'asz \cite{borgs2013left} for general Benjamini--Schramm convergent graph sequences, where $\Delta$ is a bound on the degrees of all $G_n$ in the sequence. The result of Bandyopadhyay and Gamarnik \cite{bandyopadhyay2008counting} was extended for not necessarily integer $q\geq 8\Delta$ and $w=-1$ by Ab\'ert and Hubai \cite{abert2015benjamini} also for arbitrary  Benjamini--Schramm convergent graph sequences. Csikv\'ari and Frenkel \cite{csikvari2016benjamini} showed that the same conclusion holds true for every fixed $w\geq 0$ and $q$ sufficiently large in terms of $w$ and $\Delta$. 

Our paper is essentially about the case when $q,w\leq 0$. The above papers utilize multiple different tools ranging from analytic interpolation method to cluster expansion through establishing zero-free regions or correlation decay.  Our strategy can be regarded as a combinatorial interpolation method, where we show that for large girth graphs, $T_G(x,y)$ is not very far from another graph parameter that can be analyzed more easily.
\bigskip

\noindent \textbf{Organization of this paper.} In the next section, we introduce a polynomial that we call $R_G(z)$. In some sense, the definition of this polynomial is the main idea of this paper. The polynomial $R_G(z)$ is defined via matchings and degrees of the graph $G$, and if $G$ is regular, then it simplifies to a transformation of the matching polynomial. This means that we can build on the extensive theory of matching polynomials. On the other hand, $R_G(z)$ turns out to be related to a weighted count of pseudo-forests (Corollary~\ref{R-pseudo-forests}). This makes it especially easy to compare it with $F_G(z):=z^{k(G)}T_G(z+1,1)$ that counts forests, see Section~\ref{pseudo to forests}. In Section~\ref{sec: matching polynomial}, we collected all the necessary information about the matching polynomial that we will use to prove Theorem~\ref{tutte-limit}. In Section~\ref{sec: forest-limit 1}, we prove  Theorem~\ref{tutte-limit} for the case $y=1$. Finally, in Section~\ref{sec: forest-limit general}, we prove Theorem~\ref{tutte-limit} for the case $0\leq y<1$.

\section{The polynomial $R_G(z)$}

In this section, we study a polynomial that is related to the matching polynomial of a graph $G$. Recall that a set of edges $M$ is a matching if no two edges of $M$ have common end vertices. A $k$-matching is simply a matching of size $k$.

\begin{Def}
Let
$$R_G(z)=\sum_{M\in \mathcal{M}(G)}(-z)^{|M|}\prod_{v\notin V(M)}(z+d_v-1),$$
where $d_v$ is the degree of the vertex $v$, and $\mathcal{M}(G)$ is the set of matchings of $G$ including the empty one.
\end{Def}

\begin{Rem}
Mohammadian introduced a polynomial \cite{mohammadian2019laplacian,wan2021location} that he calls the Laplacian matching polynomial defined as follows
$$\mathcal{LM}_G(z)=\sum_{M\in \mathcal{M}(G)}(-1)^{|M|}\prod_{v\notin V(M)}(z-d_v).$$
There is s striking resemblance between $R_G(z)$ and $\mathcal{LM}_G(z)$, but the Laplacian matching polynomial behaves analytically nicer. For instance, Mohammadian proved \cite{mohammadian2019laplacian} that  the Laplacian matching polynomial has only real zeros. The polynomial $R_G(z)$ does not have such nice properties, it was specifically designed to attack the problem considered in this paper.
\end{Rem}

In what follows, we first study $R_G(z)$ from the perspective of a model that we call half-edge model that is inspired by the monomer-dimer model of statistical physics. This perspective will enable us to rewrite $R_G(z+1)$ as a weighted sum of  pseudo-forests.

\subsection{Half-edge model}

In this part, we introduce the half-edge model.
 A half-edge configuration is a configuration of edges and half-edges of the graph $G$ such that each vertex of $G$ is incident to at most one edge or half-edge. For such a configuration $C$, let $C_0$ be the number of edges of $G$, where no half edge is chosen, $C_1$ is the number of edges, where exactly one half is chosen, and $C_2$ where both halves are chosen (in other words, the edge is chosen). Let
$$M_G(a_0,a_1,a_2)=\sum_Ca_0^{C_0}a_1^{C_1}a_2^{C_2}.$$

\begin{figure}[h!]
    \centering
    \begin{tikzpicture}
    \node[vertex] (u1) at (1,0) {};
    \node[vertex] (u2) at (3,0) {};
    \node[vertex] (u3) at (4,1.25) {};
    \node[vertex] (u4) at (3,2.5) {};
    \node[vertex] (u5) at (1,2.5) {};
    \node[vertex] (u6) at (0,1.25) {};
    \node[point] (v1) at (0.5,1.875) {};
    \node[point] (v2) at (3.5,1.875) {};
    \node[point] (v3) at (2,2.5) {};
    
    \draw (u1) -- (u2);
	\draw (u2) -- (u3);
	\draw (u3) -- (u4);
	\draw (u4) -- (u5);
	\draw (u5) -- (u6);
	\draw (u6) -- (u1);
	\draw[line width=3pt] (u1) -- (u2);
	\draw[line width=3pt] (u6) -- (v1);
	\draw[line width=3pt] (u5) -- (v3);
	\draw[line width=3pt] (u3) -- (v2);

    \end{tikzpicture}
    \caption{A $6$-cycle with a half-edge configuration contributing a term $a_0^2a_1^3a_2$. }
    \label{fig:my_label}
\end{figure}
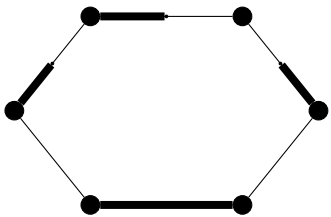

\begin{Th} \label{half-edge model} For a graph $G$ we have
$$M_G(a_0,a_1,a_2)=a_0^{|E|-n}\sum_{M\in \mathcal{M}(G)}(a_0a_2-a_1^2)^{|M|}\prod_{v\notin V(M)}(a_0+d_va_1).$$
\end{Th}

\begin{proof} See the next subsection on gauge transformation.
\end{proof}

An immediate corollary is the following.

\begin{Cor}
We have
$$M_G(z,1,-1)=z^{|E|-n}R_G(z+1).$$
\end{Cor}

Next, we prove an alternative description of $M_G(z,1,-1)$.

\begin{Def}
For an $A\subseteq E(G)$ we say that it is pseudo-forest if each of its connected components are forests or unicyclic graphs. Let $\mathcal{PF}$ be the sets of pseudo-forests. For a pseudo-forest $A$ let $c(A)$ be the number of cycles in $A$. 
\end{Def}

\begin{Lemma}
We have
$$M_G(z,1,-1)=z^{|E|-n}\sum_{k=0}^n\left(\sum_{A\in \mathcal{PF} \atop |A|=k}2^{c(A)}\right)z^{n-k}.$$
\end{Lemma}

\begin{proof} Let us fix a half-edge configuration. Let $A$ be the subset of the edges whose at least one half-edge is chosen. Let us consider a connected component of $A$, let it be $T$ and $V(T)$ be the vertices covered by $T$. Since each vertex of $T$ has at most one half-edge, we get that $|T|\leq |V(T)|$. On the other hand, $T$ is connected and so $|T|\geq |V(T)|-1$. 

If $|V(T)|=|T|$, then it contains exactly one cycle and from every edge, exactly one half is chosen. Then it is easy to see that from the unique cycle one can choose every second half-edge as this corresponds to one of the two orientations of the cycle if we regard a half-edge as an orientation. Every other half-edges are uniquely determined. So such a component corresponds to exactly two half-edge configurations.

If $|V(T)|-1=|T|$, then it contains no cycle. This can happen in two different ways. One is that if one of the vertices is not covered by a half-edge, then this uniquely determines the position of the half-edges: they are oriented towards this vertex. The other is that if two halves of an edge are chosen and every other half-edges are oriented toward this edge. In the first case, the contribution is $|V(T)|$, in the second case, the contribution is $-|T|$, so altogether we get that such a component contributes $|V(T)|-|T|=1$ to the sum. 

Hence a set $A$ contributes to the sum $M_G(z,1,-1)$ if and only if it is a pseudo-forest and its contribution is $2^{c(A)}z^{m-|A|}$. Hence
$$M_G(z,1,-1)=\sum_{A\in \mathcal{PF}}2^{c(A)}x^{|E|-|A|}=z^{|E|-n}\sum_{k=0}^n\left(\sum_{A\in \mathcal{PF} \atop |A|=k}2^{c(A)}\right)z^{n-k}.$$
\end{proof}

By comparing the previous two results we get the following corollary.

\begin{Cor} \label{R-pseudo-forests}
We have
$$R_G(z+1)=\sum_{k=0}^n\left(\sum_{A\in \mathcal{PF} \atop |A|=k}2^{c(A)}\right)z^{n-k}.$$
\end{Cor}

\subsection{Normal factor graphs and gauge transformations}

In this section, we prove Theorem~\ref{half-edge model}. The following language will be useful for us.

\begin{Def} A normal factor graph is a graph equipped with an alphabet $\mathcal{X}$ and  a function $f_v$ at each vertex: $\mathcal{H}=(V,E,\mathcal{X},(f_v)_{v\in V})$. At each edge $e$ there is a variable $x_e$ taking values from the alphabet $\mathcal{X}$. The partition function 
$$Z(\mathcal{H})=\sum_{\sigma\in \mathcal{X}^E}\prod_{v\in V}f_v(\sigma_{\partial v}),$$ 
where $\sigma_{\partial v}$ is the restriction of $\sigma$ to the the edges incident to the vertex $v$. 
\end{Def}

For instance, if $\mathcal{X}=\{0,1\}$ and 
$$f_v(\sigma_1,\dots \sigma_d)=\left\{ \begin{array}{cl} 1 & \mbox{if}\ \sum_{i=1}^d\sigma_i=1, \\ 0 & \mbox{otherwise}, \end{array} \right.$$
then $Z(\mathcal{H})$ is exactly the number of perfect matchings of the underlying graph. 

Let $\mathcal{H}=(V,E,(f_v)_{v\in V})$ be a normal factor graph with alphabet $\mathcal{X}$. We will show that it is possible to introduce a new normal factor graph $\widehat{\mathcal{H}}$ on the same graph with new functions $\widehat{f_v}$ such that $Z(\widehat{\mathcal{H}})=Z(\mathcal{H})$. As we will see, sometimes it will be more convenient to study the new normal factor graph $\widehat{\mathcal{H}}$. 

Let $\mathcal{Y}$ be a new alphabet, and for each edge $(u,v)\in E$ let us introduce two new matrices, $G_{uv}$ and $G_{vu}$ of size $\mathcal{Y}\times \mathcal{X}$. The new variables will be denoted by $\tau$, the old ones by $\sigma$. Let
$$\widehat{f_v}(\tau_{vu_1},\dots ,\tau_{vu_k})=\sum_{\sigma_{vu_1},\dots ,\sigma_{vu_k}}\left(\prod_{u_i \in N(v)}G_{vu_i}(\tau_{vu_i},\sigma_{vu_i})\right)f_v(\sigma_{vu_1},\dots ,\sigma_{vu_k}).$$
This way we defined the functions $\widehat{f_v}$ of $\widehat{\mathcal{H}}$. 
\medskip

This transformation is called a gauge transformation. In computer science, this method was introduced by L. Valiant under the name holographic reduction \cite{valiant2008holographic,valiant2006accidental,valiant2002quantum,valiant2002expressiveness}. In statistical physics, it was developed by M. Chertkov and V. Chernyak under the name gauge transformation \cite{chertkov2006loop2,chertkov2006loop1}. Wainwright, Jaakola, Willsky had a related idea under the name reparametrization \cite{wainwright2003tree}, but it is not easy to see the connection. In the different cases the scope was slightly different, L. Valiant used it as a reduction method for computational complexity of counting problems. This line of research was extended in a series of papers of Jin-Yi Cai and his coauthors, see Jin-Yi Cai's book \cite{cai2017complexity} and the papers \cite{cai2007holographic,cai2008basis,cai2007symmetric,cai2008holographic, cai2011holographic, cailu2008holographic} and references therein. M. Chertkov and V. Chernyak \cite{chertkov2006loop2,chertkov2006loop1} studied the so-called Bethe--approximation through gauge transformations. We simply use it as a method of proving the identity of Theorem~\ref{half-edge model}.

The following theorem is due to Chertkov and Chernyak \cite{chertkov2006loop2,chertkov2006loop1} and independently Valiant \cite{valiant2008holographic}. 

\begin{Th} If for each edge $(u,v)\in E$ we have $G^T_{uv}G_{vu}=\mathrm{Id}_{\mathcal{X}}$, then $Z(\widehat{\mathcal{H}})=Z(\mathcal{H})$. 
\end{Th}

\begin{proof}
Let us start to compute 
$Z(\widehat{\mathcal{H}})=\sum_{\tau\in \mathcal{Y}^E}\prod_{v\in V}\widehat{f_v}(\tau_{\partial v})$:
$$Z(\widehat{\mathcal{H}})=\sum_{\tau\in \mathcal{Y}^E}\prod_{v\in V}\left[\sum_{\sigma_{vu_1},\dots ,\sigma_{vu_k}}\left(\prod_{u_i \in N(v)}G_{vu_i}(\tau_{vu_i},\sigma_{vu_i})\right)f_v(\sigma_{vu_1},\dots ,\sigma_{vu_k})\right].$$
If we expand it will have terms $\prod_{v\in V}f_v(\sigma_{vu_1},\dots ,\sigma_{vu_k})$ with some coefficients. A priori it can occur that these terms are incompatible in the sense that $\sigma_{uv}\neq \sigma_{vu}$. As we will see that the role of the conditions on $G_{uv}$ is exactly to ensure that if there is an edge $(u,v)\in E$ with 
$\sigma_{uv}\neq \sigma_{vu}$, then the coefficient is $0$, and if all edges are compatible, then the coefficient is $1$. Indeed, the coefficient is
$$\sum_{\tau\in \mathcal{Y}^E}\prod_{v\in V}\prod_{u_i \in N(v)}G_{vu_i}(\tau_{vu_i},\sigma_{vu_i}).$$
Note that $\tau_{uv}=\tau_{vu}$ for each edge, and this variable appears only at the vertices $u$ and $v$, and nowhere else. Hence
$$\sum_{\tau\in \mathcal{Y}^E}\prod_{v\in V}\prod_{u_i \in N(v)}G_{vu_i}(\tau_{vu_i},\sigma_{vu_i})=\prod_{(u,v)\in E}\left(\sum_{\tau_{uv}}G_{uv}(\tau_{uv},\sigma_{uv})G_{vu}(\tau_{vu},\sigma_{vu})\right)=$$
$$=\prod_{(u,v)\in E}\left(\sum_{\tau_{uv}}G^T_{uv}(\sigma_{uv},\tau_{vu})G_{vu}(\tau_{vu},\sigma_{vu})\right)=\prod_{(u,v)\in E}(G^T_{uv}G_{vu})_{\sigma_{uv},\sigma_{vu}}=
\prod_{(u,v)\in E}(\mathrm{Id})_{\sigma_{uv},\sigma_{vu}}.$$
Hence this is only non-zero if $\sigma_{uv}=\sigma_{vu}$ for each edge $(u,v)\in E(G)$, and then this coefficient is $1$.
\end{proof}

Now we are ready to prove Theorem~\ref{half-edge model}.

\begin{proof}[Proof of Theorem~\ref{half-edge model}]
First we turn the half-edge model into a normal factor graph. First let $\mathrm{Sub}(G)$ be the subdivision of the graph $G$, that is, we subdivide each edge of the graph with a new vertex. Then let $\mathcal{H}$ be the normal factor graph with underlying graph $\mathrm{Sub}(G)$, alphabet $\mathcal{X}=\{0,1\}$ and the following functions. For the original vertices of the graph let 
$$f_v(\sigma_1,\dots ,\sigma_{d_v})=\left\{ \begin{array}{cl} 1 & \mbox{if}\ \sum_{i=1}^{d_v}\sigma_i\leq 1, \\ 0 & \mbox{otherwise}. \end{array} \right.$$
For the new vertices 
$$f_e(\sigma_1,\sigma_2)=a_{\sigma_1+\sigma_2}.$$
Later it will be more convenient to work with the matrix 
$$F_e:=\left( \begin{array}{cc} a_0 & a_1 \\ a_1 & a_2 \end{array}\right).$$
Then $Z(\mathcal{H})=M_G(a_0,a_1,a_2)$.
\medskip

Next we use the gauge theory. For each edge $e=(u,v)\in E(G)$ we introduce two matrices:
$G_{eu}=G_{ev}=G_1$ and $G_{ue}=G_{ve}=G_2$, where
$$G_1:=\left( \begin{array}{cc} 1 & 0 \\ -\frac{a_1}{a_0} & 1 \end{array}\right)\ \ \ \mbox{and} \ \ \ G_2:=\left( \begin{array}{cc} 1 & \frac{a_1}{a_0} \\ 0 & 1 \end{array}\right).$$
Observe that $G_2^TG_1=\mathrm{Id}$. First let us compute $\widehat{f_e}(\tau_1,\tau_2)$:
\begin{align*}\widehat{f_e}(\tau_1,\tau_2)&=\sum_{\sigma_1,\sigma_2}G_1(\tau_1,\sigma_1)G_1(\tau_2,\sigma_2)f_e(\sigma_1,\sigma_2)\\
&=\sum_{\sigma_1,\sigma_2}G_1(\tau_1,\sigma_1)f_e(\sigma_1,\sigma_2)G^T_1(\sigma_2,\tau_2)\\
&=(G_1F_eG^T_1)_{\tau_1,\tau_2}
\end{align*}
Here we have
$$G_1F_eG^T_1=\left( \begin{array}{cc} a_0 & 0 \\ 0 & \frac{a_0a_2-a_1^2}{a_0} \end{array}\right).$$
Next let us compute $\widehat{f_v}(\tau_1,\dots ,\tau_{d_v})$, where $d_v$ is the degree of the vertex $v$. By definition
$$\widehat{f_v}(\tau_1,\dots ,\tau_{d_v})=\sum_{\sigma_1,\dots ,\sigma_{d_v}}\prod_{i=1}^{d_v}G_2(\tau_i,\sigma_i)f_v(\sigma_1,\dots ,\sigma_{d_v}).$$
In particular,
$$\widehat{f_v}(0,\dots ,0)=G_2(0,0)^{d_v}+d_vG_2(0,0)^{d_v-1}G_2(0,1)=1+d_v\frac{a_1}{a_0}=\frac{a_0+d_va_1}{a_0}$$
and 
$$\widehat{f_v}(1,0,\dots ,0)=G_2(0,0)^{d_v-1}G_2(1,1)=1,$$
and
$$\widehat{f_v}(\tau_1,\dots ,\tau_{d_v})=0$$
if $\sum_{i=1}^{d_v}\tau_i\geq 2$. To see the last two evaluations observe that since $G_2(1,0)=0$, a non-zero term implies  that if $\tau_i=1$, then $\sigma_i=1$ too. Hence in case of a non-zero term $\sum_{i=1}^{d_v}\sigma_i\geq \sum_{i=1}^{d_v}\tau_i$. Since $f_v(\sigma_1,\dots \sigma_{d_v})=0$ if $\sum_{i=1}^{d_v}\sigma_i\geq 2$ we only need to check a few terms. 

Now let us compute
$$Z(\widehat{\mathcal{H}})=\sum_{\tau\in \mathcal{Y}^{E'}}\prod_{v\in V}\widehat{f_v}(\tau_{\partial v})\prod_{e\in E(G)}\widehat{f_e}(\tau_{\partial e}).$$
Since $\widehat{f_e}(0,1)=\widehat{f_e}(1,0)=0$, the non-zero terms correspond to the edge set of the original graph $G$. Furthermore, since $\widehat{f_v}(\tau_1,\dots ,\tau_{d_v})=0$ if $\sum_{i=1}^{d_v}\tau_i\geq 2$ this edge set has to be a matching. The contribution of a matching $M$ is
$$\left(\frac{a_2a_0-a_1^2}{a_0}\right)^{|M|}a_0^{|E|-|M|}\prod_{v\notin V(M)}\frac{a_0+d_va_1}{a_0}=a_0^{|E|-n}(a_2a_0-a_1^2)^{|M|}\prod_{v\notin V(M)}(a_0+d_va_1).$$
Hence
$$Z(\widehat{\mathcal{H}})=a_0^{|E|-n}\sum_{M\in \mathcal{M}(G)}(a_0a_2-a_1^2)^{|M|}\prod_{v\notin V(M)}(a_0+d_va_1).$$
Since $Z(\mathcal{H})=Z(\widehat{\mathcal{H}})$ the claim of the theorem follows.
\end{proof}

\section{From pseudo-forests to forests} \label{pseudo to forests}

Let us introduce the polynomial
$$F_G(z)=\sum_{k=0}^nf_k(G)z^{n-k},$$
where $f_k(G)$ denotes the number of forests with exactly $k$ edges. It turns out that if $T_G(x,y)$ denotes the Tutte polynomial of the graph $G$, then
$$F_G(z)=z^{k(G)}T_G(z+1,1).$$
Since a forest is a pseudo-forest without any cycle we immediately have
$$F_G(z)\leq R_G(z+1)$$
for positive $z$. The following theorem shows that for (essentially) large girth graphs the two quantities are not two far from each other. It is known that random regular graphs contain at most $O_d(1)$ cycles of length $k$ for every fixed $k$ with very high probability as the random variable $X_k$ counting the number of $k$-cycles has asymptotically Poisson distribution with parameter $\frac{(d-1)^k}{k}$ (see \cite{mckay2004short} and the references therein).

\begin{Lemma} \label{comparison}
Let $G$ be a graph on $n$ vertices with average degree $\overline{d}$ such that it contains at most $L$ cycles of  length at most $g-1$. Then
$$\left(1+\frac{g \overline{d}}{z}\right)^{-L-n/g}R_G(z+1)\leq F_G(z)\leq R_G(z+1).$$

\end{Lemma}

\begin{proof}
We have seen that the inequality $F_G(z)\leq R_G(z+1)$ is trivial, so we only need to prove the other inequality.
Let $F$ be a forest with connected components $T_1,\dots, T_k$, where $k=k(F)$. For each $T_j$ let $V(T_j)$ be its induced vertex set. Let $E[V(T_j)]$ be the subset of edges of $G$ induced by $V(T_j)$. It contains the edges of $T_j$, but it may contain other edges. 

Furthermore, let $\ell(F)$ be the number of components that induces a graph with some cycles, that is, the number of components, where $|E[V(T_j)]|>|V(T_j)|-1$. We can embed the forest $F$ into a pseudo-forest in
$$\prod_{j=1}^{\ell(F)}(1+(|E[V(T_j)]|-|V(T_j)|+1))$$
ways with the same induced connected components. Indeed, at each component $T_j$ we may add no edges, or we may add one of the $|E[V(T_j)]|-|V(T_j)|+1$ edges. If we  take into account the weights $2^{c(A)}z^{-|A|}$ for a pseudo-forest, then we get that
$$\sum_{A\in \mathcal{PF}(G)}2^{c(A)}z^{-|A|}\leq \sum_{F\in \mathcal{F}(G)}z^{-|F|}\prod_{j=1}^{\ell(F)}\left(1+\frac{2}{z}(|E[V(T_j)]|-|V(T_j)|+1)\right).$$

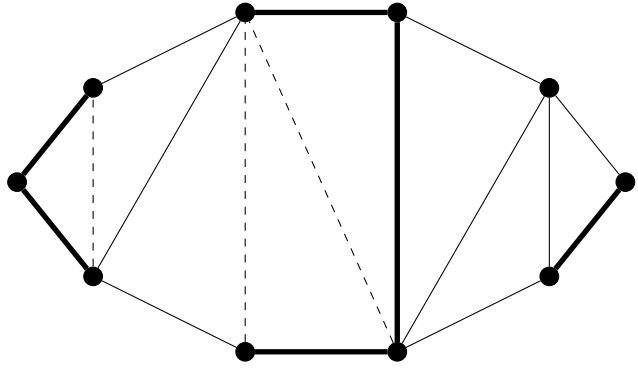
\begin{figure}[h!]
    \centering
    \begin{tikzpicture}
    \node[vertex] (u1) at (1,0) {};
    \node[vertex] (u2) at (3,-1) {};
    \node[vertex] (u3) at (5,-1) {};
    \node[vertex] (u4) at (7,0) {};
    \node[vertex] (u5) at (8,1.25) {};
    \node[vertex] (u6) at (7,2.5) {};
    \node[vertex] (u7) at (5,3.5) {};
    \node[vertex] (u8) at (3,3.5) {};
    \node[vertex] (u9) at (1,2.5) {};
    \node[vertex] (u10) at (0,1.25) {};
    
    \draw (u1) -- (u2);
	\draw[line width=2pt] (u2) -- (u3);
	\draw (u3) -- (u4);
	\draw[line width=2pt] (u4) -- (u5);
	\draw (u5) -- (u6);
	\draw (u6) -- (u7);
	\draw[line width=2pt] (u7) -- (u8);
	\draw (u8) -- (u9);
	\draw[line width=2pt] (u9) -- (u10);
	\draw[line width=2pt] (u10) -- (u1);
	\draw[dashed] (u1) -- (u9);
	\draw (u1) -- (u8);
	\draw[dashed] (u2) -- (u8);
	\draw[line width=2pt] (u3) -- (u7);
	\draw[dashed] (u3) -- (u8);
	\draw (u3) -- (u6);
	\draw (u4) -- (u6);

    \end{tikzpicture}
    \caption{A forest $F$ is depicted with thick edges. The dashed edges may be added to create a pseudo-forest with the same connected components, but in the middle of the picture at most one of the two dashed edges can be added.}
    \label{fig:my_label}
\end{figure}

Note that we have no equality in general since we get the same pseudo-forest from many forests. Note that if $|V(T_j)|<g$, then $E[V(T_j)]$ cannot contain a cycle with the exception of at most $L$ components. Thus $\ell(F)\leq L+n/g$. Then
\begin{align*}
\prod_{j=1}^{\ell(F)}\left(1+\frac{2}{z}(|E[V(T_j)]|-|V(T_j)|+1)\right)&\leq \left(\frac{1}{\ell(F)}\sum_{j=1}^{\ell(F)}\left(1+\frac{2}{z}(|E[V(T_j)]|-|V(T_j)|+1)\right)\right)^{\ell(F)}\\
&\leq \left(\frac{1}{\ell(F)}\sum_{j=1}^{\ell(F)}\left(1+\frac{2}{z}|E[V(T_j)]|\right)\right)^{\ell(F)}\\
&\leq \left(1+\frac{2|E|}{z\ell(F)}\right)^{\ell(F)}
\end{align*}
Since $\ell(F)\leq L+n/g$ and the function $(1+c/t)^t$ is monotone increasing for positive $t$ for every $c$ we have
$$\left(1+\frac{2|E|}{z\ell(F)}\right)^{\ell(F)}\leq \left(1+\frac{2|E|}{z(L+n/g)}\right)^{L+n/g}\leq\left(1+\frac{g \overline{d}}{z}\right)^{L+n/g}.$$
Hence 
$$R_G(z+1)=\sum_{A\in \mathcal{PF}(G)}2^{c(A)}z^{n-|A|}\leq \left(1+\frac{g \overline{d}}{z}\right)^{L+n/g}F_G(z).$$

\end{proof}

\section{Matching polynomial} \label{sec: matching polynomial}

In this section, we collect a few things about the matching polynomial. The matching polynomial of the graph $G$ is defined as follows. Let 
$$\mu_G(z)=\sum_{k=0}^{n/2}(-1)^km_k(G)z^{n-2k},$$
where $m_k(G)$ denotes the number of matchings of size $k$.
Note that for a $d$-regular graph $G$ we have
$$R_G(z)=\sum_{M\in \mathcal{M}(G)}(-z)^{|M|}(d+z-1)^{n-2|M|}=z^{n/2}\mu_G\left(\frac{d-1+z}{\sqrt{z}}\right).$$
The following theorem about the matching polynomial is fundamental for us.

\begin{Th}[Heilmann and Lieb \cite{heilmann1972theory}] \label{th: HeiLie}
All zeros of the matching polynomial $\mu_G(z)$ are real. Furthermore, if the largest degree $\Delta$ satisfies $\Delta \geq 2$, then all zeros lie in the interval $(-2\sqrt{\Delta-1},2\sqrt{\Delta-1})$.
\end{Th}

Let $\mu_G(z)=\prod_{i=1}^n(z-\alpha_i)$, and
$s_k(G)=\frac{1}{v(G)}\sum_{i=1}^n \alpha_i^k$.
The quantities $s_k(G)$ have a combinatorial meaning. They count the so-called tree-like walks.

\begin{Def} Let $G$ be graph with a given vertex $u$. The \textit{path-tree} $T(G,u)$ is defined as follows. The vertices of $T(G,u)$ are the paths in $G$ which start at the  vertex $u$ and two paths joined by an edge if one of them is a one-step extension of the other.
\end{Def}

\begin{figure}
\begin{center}
\includegraphics[scale=0.8]{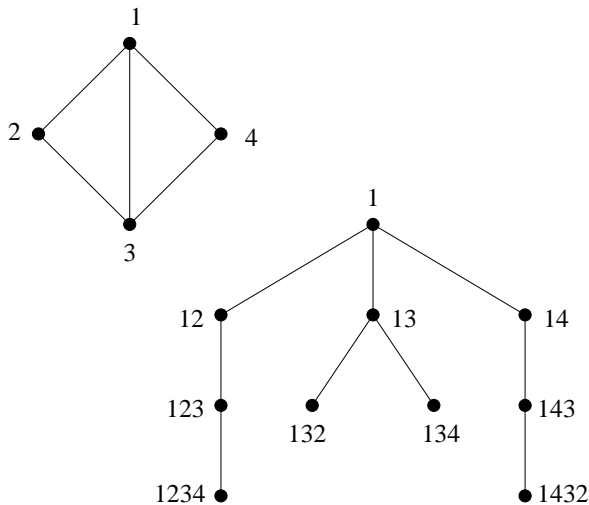}
\caption{A path-tree from the vertex $1$.}  
\end{center}
\end{figure}

It turns out that $s_k(G)$ counts certain walks called tree-like walks.

\begin{Def}
A closed tree-like walk of length $\ell$ is a closed walk on $T(G,u)$ of length $\ell$ starting and ending at the vertex $u$.
\end{Def}

Note that a priori the tree-like walk is on the path-tree, although one can make a correspondence with certain walks on the graph itself. Indeed, a walk in the tree $T(G,u)$ from $u$ can be imagined as follows. Suppose that in the graph $G$ a worm is sitting at the vertex $u$ at the beginning. Then at each step the worm can either grow or pull back its head. When it grows it can move its head to a neighboring unoccupied vertex while keeping its tail at vertex $u$.  At each step the worm occupies a path in the graph $G$. A closed walk in the tree $T(G,u)$ from $u$ to $u$ corresponds to the case when at the final step the worm occupies only vertex $u$.

\begin{Lemma}[Godsil \cite{godsil1993algebraic}] \label{tree-like walks}
The number of all closed tree-like walks of length $\ell$ is exactly $s_{\ell}(G)$.
\end{Lemma}

Note that we can introduce $s_k(\mathbb{T}_d,o)$ this way: this is simply the number of closed walks from a root vertex $o$ of the infinite $d$-regular tree of length $k$. 

Let $\rho^{m}_G$ be the uniform measure supported on these roots, that is, given a function $f:\mathbb{C}\to\mathbb{C}$ we have
$$\int f(z) d\rho^{m}_G(z)=\frac{1}{v(G)}\sum_{i=1}^{v(G)}f(\alpha_i).$$
In particular, 
$$\int z^k d\rho^{m}_G(z)=\frac{1}{v(G)}\sum_{i=1}^{v(G)}\alpha_i^k=\frac{s_k(G)}{v(G)}.$$
There is a measure $\rho_{\mathrm{KM}}$ called Kesten-McKay measure for which 
$$s_k(\mathbb{T}_d,o)=\int z^k d\rho_{\mathrm{KM}}(z).$$
The Kesten-McKay measure has an explicit density function
$$\frac{d\sqrt{4(d-1)-z^2}}{2\pi (d^2-z^2)}\cdot \textbf{1}_{(-\omega,\omega)},$$
where $\omega=2\sqrt{d-1}$ and $\textbf{1}_{(-\omega,\omega)}$ is the characteristic function of the interval.
The following lemma is a special case of a more general theorem of Csikv\'ari and Frenkel  \cite{csikvari2016benjamini} about Benjamini--Schramm convergent graph sequences, see also the papers \cite{abert2016matchings,abert2015matching}.

\begin{Lemma} \label{conv-measures}
If $(G_n)_n$ is a sequence of $d$-regular graphs with $g(G_n)\to \infty$, then the measures $\rho^{m}_{G_n}$  converge weakly to the Kesten-McKay measure, that is, for every bounded continuous function $f$ we have
$$\lim_{n\to \infty}\int f(z) d\rho^{m}_{G_n}(z)=\int f(z) d\rho_{\mathrm{KM}}(z).$$
\end{Lemma}

The proof of Lemma~\ref{conv-measures} is particularly simple since every continuous function on the interval $[-2\sqrt{d-1},2\sqrt{d-1}]$ can be approximated by a polynomial in sup norm. On the other hand, if $k<g(G)/2$, then
$$\int z^k d\rho^{m}_{G}(z)=\frac{s_k(G)}{v(G)}=s_k(\mathbb{T}_d,o)=\int z^k d\rho_{\mathrm{KM}}(z).$$

Finally we need the evaluation of certain integral along the Kesten-McKay measure. 

\begin{Lemma} \label{integral evaluation} For $0\leq z\leq d-1$ we have
$$z^{1/2}\exp\left(\int \ln\left(\frac{d+z-1}{\sqrt{z}}-t\right) d\rho_{\mathrm{KM}}(t)\right)=(d-1)\left(\frac{(d-1)^2}{(d-1)^2-z}\right)^{d/2-1}.$$
For $z>d-1$ we have
$$z^{1/2}\exp\left(\int \ln\left(\frac{d+z-1}{\sqrt{z}}-t\right) d\rho_{\mathrm{KM}}(t)\right)=z\left(1+\frac{1}{z-1}\right)^{d/2-1}.$$
\end{Lemma}

The proof of Lemma~\ref{integral evaluation} is actually a simple consequence of the following theorem of Mckay (see Lemma 3.5 of \cite{mckay1983spanning}).

\begin{Lemma} Let $\omega=2\sqrt{d-1}$, and $|\gamma|<\frac{1}{\omega}$. Then
$$\int_{-\omega}^{\omega}\ln(1-\gamma z)\ d\rho_{\mathrm{KM}}(z) = \ln \left(\frac{1}{\eta}\left(\frac{d-1}{d-\eta}\right)^{d/2-1}\right),$$
where 
$$\eta=\frac{1-(1-4(d-1)\gamma^2)^{1/2}}{2(d-1)\gamma^2}.$$
\end{Lemma}

\begin{proof}[Proof of Lemma~\ref{integral evaluation}]
Let $\gamma=\frac{\sqrt{z}}{d+z-1}$. Note that $|\gamma|<\frac{1}{\omega}$ if $z\neq d-1$. The result follows by continuity if $z=d-1$. Then
$$\ln\left(\frac{d+z-1}{\sqrt{z}}-t\right)=\ln \left(\frac{d+z-1}{\sqrt{z}}\right)+\ln(1-\gamma t).$$
Note that
$$1-4(d-1)\gamma^2=\frac{(d-1-z)^2}{(d+z-1)^2}$$
and so
$$\frac{1-(1-4(d-1)\gamma^2)^{1/2}}{2(d-1)\gamma^2}=\left\{
\begin{array}{ll}
\frac{d+z-1}{d-1} &\mbox{if}\ \ 0\leq z\leq d-1,\\
\frac{d+z-1}{z} &\mbox{if}\ \  d-1\leq z.\\
\end{array}\right.$$
The rest is simple substitution.

\end{proof}

\section{Proof of Theorems~\ref{forest-limit} and \ref{tutte-limit} for $y=1$} \label{sec: forest-limit 1}

In this part, we prove Theorems~\ref{forest-limit} and \ref{tutte-limit} for $y=1$.

\begin{proof}[Proof of Theorems~\ref{forest-limit} and \ref{tutte-limit}]
Let $x=z+1$. First we assume that $x>1$, that is, $z>0$. The claim will follow for $x=1$ by continuity, see the end of the proof.
Recall that $F_G(z)=z^{k(G)}T_G(z+1,1)$. 
By Theorem~\ref{comparison} we have
$$\left(1+\frac{g \overline{d}}{z}\right)^{-L-n/g}R_G(z+1)\leq F_G(z)\leq R_G(z+1)$$
if $G$ contains at most $L$ cycles of length at most $g-1$. Thus
$$\lim_{n\to \infty}F_{G_n}(z)^{1/v(G_n)}=\lim_{n\to \infty}R_{G_n}(z+1)^{1/v(G_n)}.$$
In case of a random $d$-regular graph sequence $(G_n)_n$ we use the fact that $L=O_{g,d}(1)=o(v(G_n))$ asymptotically almost surely.\\
Since
$$R_G(z)=\sum_{M\in \mathcal{G}}(-z)^{|M|}(d+z-1)^{n-2|M|}=z^{v(G)/2}\mu_G\left(\frac{d-1+z}{\sqrt{z}}\right)$$
for $d$-regular graphs we have
$$\frac{1}{v(G)}\ln R_G(z+1)=\frac{1}{2}\ln(z+1)+\frac{1}{v(G)}\ln \mu_G\left(\frac{d+z}{\sqrt{z+1}}\right).$$
Let us introduce the notation $u:=\frac{d+z}{\sqrt{z+1}}$. Note that
$$u=\frac{d+z}{\sqrt{z+1}}=\frac{(d-1)+z+1}{\sqrt{z+1}}=\frac{d-1}{\sqrt{z+1}}+\sqrt{z+1}\geq 2\sqrt{d-1}$$
and we have strict inequality if $z+1\neq d-1$.
Thus
$$\frac{1}{v(G)}\ln R_G(z+1)=\frac{1}{2}\ln(z+1)+\frac{1}{v(G)}\ln \mu_G(u).$$
Here 
$$\frac{1}{v(G)}\ln \mu_G(u)=\frac{1}{v(G)}\sum_{i=1}^n\ln(u-\alpha_i)=\int \ln(u-t)\, d\rho^{m}_G(t).$$
If $z+1\neq d-1$, then the function $\ln(u-t)$ is a bounded continuous function on the interval $(-2\sqrt{d-1},2\sqrt{d-1})$ and $\rho^{m}_{G_n}$  converges weakly to the Kesten-McKay measure. Hence
$$\lim_{n\to \infty}\int \ln(u-t)\, d\rho^{m}_{G_n}(t)=\int \ln(u-t)\, d\rho_{\mathrm{KM}}(t).$$
Putting these facts together we get that
$$\lim_{n\to \infty}F_{G_n}(z)^{1/v(G_n)}=(z+1)^{1/2}\exp\left(\int \ln\left(\frac{d+z}{\sqrt{z+1}}-t\right) d\rho_{\mathrm{KM}}(t)\right)=t_d(z+1),$$
where in the last step we used Lemma~\ref{integral evaluation}.
If $z+1=d-1$ then the claim is still true since $F_G(z)$ is a continuous, monotone increasing function just as the function 
$t_d(z+1)$. So if for every $z\neq d-2$ we have $\lim_{n\to \infty}F_{G_n}(z)^{1/v(G_n)}=t_d(z+1)$, then it is also true at $z=d-2$. Hence
$$\lim_{n\to \infty} ((x-1)^{k(G_n)}T_{G_n}(x,1))^{1/v(G_n)}=t_d(x)$$
for all $x>1$, that is, 
$$\lim_{n\to \infty} T_{G_n}(x,1)^{1/v(G_n)}=t_d(x).$$
Finally, for $x=1$ we can conclude as follows.
Since the coefficients of $T_G(x,1)$ are non-negative and its degree is $v(G)-k(G)$ we get that for $x>1$ we have
$$T_G(1,1)\leq T_G(x,1)\leq x^{v(G)}T_G(1,1).$$
Hence 
$$\limsup_{n\to \infty}T_{G_n}(1,1)^{1/v(G_n)}\leq \limsup_{n\to \infty}T_{G_n}(x,1)^{1/v(G_n)}=t_d(x)$$
and
$$\liminf_{n\to \infty}T_{G_n}(1,1)^{1/v(G_n)}\geq \liminf_{n\to \infty}(x^{-v(G_n)}T_{G_n}(x,1))^{1/v(G_n)}=\frac{t_d(x)}{x}.$$
Since $\lim_{x\searrow 1}t_d(x)=\lim_{x\searrow 1}\frac{t_d(x)}{x}=t_d(1)$ we get that
$$\lim_{n\to \infty}T_{G_n}(1,1)^{1/v(G_n)}=t_d(1).$$
Finally, in the particular case when $x=2$  we get that
$$\lim_{n\to \infty}F(G_n)^{1/v(G_n)}=t_d(2)= (d-1)\left(\frac{(d-1)^2}{d^2-2d-1}\right)^{d/2-1}.$$
\end{proof}

\section{Proof of Theorem  \ref{tutte-limit} for  $0\leq y<1$} \label{sec: forest-limit general}

In this part, we prove Theorems~\ref{forest-limit} and \ref{tutte-limit} for arbitrary $0\leq y< 1$.

\begin{Def}
Given a graph $G$ and an ordering of the edges, a path is called a broken cycle if it can be obtained from a cycle by deleting the edge with highest index.
\end{Def}

The following lemma is equivalent with Whitney's characterization of the coefficients of the chromatic polynomial \cite{whitney1932logical}, equivalently it can be deduced from Tutte's original definition of his polynomial \cite{tutte1954contribution}. 

\begin{Lemma}
Let $G$ be a connected graph on $n$ vertices with an arbitrary ordering of the edges. Let
$$z^{k(G)}T_G(z+1,0)=\sum_{k=1}^{n}c_kz^{n-k}.$$
Then $c_k$ counts the number of edge sets of size $k$ containing no broken cycle.
\end{Lemma}

We will also need the following form of the FKG-inequality \cite{fortuin1971correlation}.

\begin{Lemma}[Fortuin, Kasteleyn, Ginibre \cite{fortuin1971correlation}]
Suppose that $\mu$ is the uniform measure on $[0,1]^N$, and $X_1,\dots ,X_t$ are non-negative monotone increasing functions in the sense that if $x_i\geq x_i'$ for $i=1,\dots ,N$, then for $1\leq j\leq t$ we have
$$X_j(x_1,\dots ,x_n)\geq X_j(x_1',\dots ,x_n').$$
Then
$$\E_{\mu}\left[\prod_{j=1}^tX_j\right]\geq \prod_{j=1}^t\E_{\mu}[X_j].$$
\end{Lemma}

\begin{Rem}
The above special case of the FKG-inequality might be non-standard, because it is mostly stated for finite distributive lattices, but $[0,1]^N$ can be approximated by the finite distributive lattices $\left\{0,\frac{1}{M},\frac{2}{M},\dots ,1\right\}^N$, so the above version follows from the usual versions.
\end{Rem}

\begin{Lemma} \label{lem: tress vs bcf trees}
Let $G$ be a graph with $n$ vertices, $m$ edges and at most $L$ cycles of length at most $g-1$. Let $c_k$ be the number of edge sets with exactly $k$ edges with no broken-cycle, and let $f_k$ be  the number of forests with exactly $k$ edges. Then
$$c_k\geq \left(\frac{2}{3}\right)^L\left(1-\frac{1}{g}\right)^{m-n+k(G)-L}f_k.$$
\end{Lemma}

\begin{proof}
We use the fact that $c_k$ is independent of the ordering, so we consider a random ordering of the edges and bound the probability that a fixed spanning forest has no broken cycle with respect to this ordering. We create the random ordering as follows: for each edge $e$ we choose a uniform random number $x_e$ from the interval $[0,1]$. The numbers $x_e$ then determines an ordering of the edges. The probability that two numbers are equal is $0$. 

Let us fix a  spanning forest $F$. Suppose that the  edge $e\notin E(F)$ connects two vertices in the same connected component of $F$. Let $M_F$ be the set of such edges. Finally, let $C_{F,e}$ be the cycle determined by the edge $e$ and the unique path in $F$ determined by the end points of $e$. For $e$ let $y_e=1-x_e$. Furthermore, let
$$X_{F,e}(\{x_f\}_{f\in E(F)},y_e)=\left\{ \begin{array}{ll}
1 & \mbox{if}\ \max_{f\in C_{F,e}\setminus e}x_f\geq 1-y_e,\\
0 & \mbox{if}\ \max_{f\in C_{F,e}\setminus e}x_f< 1-y_e.
\end{array} \right.$$
Furthermore, let
$$X_F=\prod_{e\in M_F}X_{F,e}.$$
Note that
$\E X_{F,e}=1-\frac{1}{|C_{F,e}|},$
and $X_{F,e}$ are monotone increasing functions in terms of the variables $\{x_f\}_{f\in E(F)}$ and $\{y_e\}_{e\in M_F}$. Hence by the FKG-inequality we have 
$$\E X_F=\E\left[ \prod_{e\in M_F}X_{F,e}\right]\geq \prod_{e\in M_F}\E X_{F,e}=\prod_{e\in M_F}\left(1-\frac{1}{|C_{F,e}|}\right).$$
Since $G$ contains at most $L$ cycles of length at most $g-1$ we have
$$\prod_{e\in M_F}\left(1-\frac{1}{|C_{F,e}|}\right)\geq \left(\frac{2}{3}\right)^L\left(1-\frac{1}{g}\right)^{|M_F|-L}.$$
Note that $|M_F|\leq m-n+k(G)$. 
Thus
$$\E X_F\geq \left(\frac{2}{3}\right)^L\left(1-\frac{1}{g}\right)^{m-n+k(G)-L}.$$
Clearly, $X_F=1$ if and only if the spanning forest $F$ contains no broken cycle  with respect to the random ordering. Hence
$$c_k=\E \left(\sum_{F\in \mathcal{F}(G) \atop |F|=k}X_F\right)=\sum_{F\in \mathcal{F}(G) \atop |F|=k}\E X_F\geq \left(\frac{2}{3}\right)^L\left(1-\frac{1}{g}\right)^{m-n+k(G)-L}f_k.$$
Note that in the first expected value we actually considered a random variable that is constant!
\end{proof}

\begin{proof}[Proof of Theorems of \ref{tutte-limit} for general $y$.]
Let $x=z+1$. Then $z\geq 0$.
Let $L=L(G,g)$ denote the number of cycles of length at most $g-1$ in $G$.
By Lemma~\ref{lem: tress vs bcf trees} we have 
$$\sum_{k=1}^nc_kz^{n-k}\geq \left(\frac{2}{3}\right)^L\left(1-\frac{1}{g}\right)^{m-n+k(G)-L}\sum_{k=1}^nf_kz^{n-k}\geq \left(\frac{2}{3}\right)^L\left(1-\frac{1}{g}\right)^{m}\sum_{k=1}^nf_kz^{n-k}.$$
This is equivalent with
$$T_G(z+1,0)^{1/v(G)}\geq \left(\frac{2}{3}\right)^{L/v(G)}\left(1-\frac{1}{g}\right)^{d/2}T_G(z+1,1)^{1/v(G)}.$$
The coefficients of the Tutte polynomial are non-negative, thus
$$T_G(z+1,0)\leq T_G(z+1,y)\leq T_G(z+1,1).$$
This shows that if $\lim_{n\to \infty}\frac{L(G_n,g)}{v(G_n)}=0$ for every $g$, then
$$\lim_{n\to \infty}T_{G_n}(z+1,1)^{1/v(G)}=\lim_{n\to \infty}T_{G_n}(z+1,y)^{1/v(G)}\lim_{n\to \infty}T_{G_n}(z+1,0)^{1/v(G)}.$$
\end{proof}

\bibliography{forest_reference}
\bibliographystyle{plain}

\end{document}